\documentclass[reqno,12pt]{amsart}
\usepackage{amsmath}
\usepackage{arydshln}%
\allowdisplaybreaks[3]
\numberwithin{equation}{section}
\numberwithin{equation}{section}

\def\proof{\indent{\em Proof.\quad}}
\def\endproof{\hfill\hbox{$\sqcup$}\llap{\hbox{$\sqcap$}}\medskip}
\newtheorem{thm}{\indent Theorem}[section]
\newtheorem{cor}[thm]{\indent Corollary}
\newtheorem{lem}[thm]{\indent Lemma}
\newtheorem{prop}[thm]{\indent Proposition}
\newtheorem{dfn}{{\indent\bf Definition}}[section]
\newtheorem{rmk}{{\indent\bf Remark}}[section]
\newtheorem{expl}{{\indent\bf Example}}[section]

\newcommand{\mb}{\mbox}

\newcommand{\ima}{\sqrt{-1}}

\newcommand{\td}{\tilde}

\newcommand{\fr}{\frac}

\newcommand{\edd}{\end{document}}
\newcommand{\be}{\begin{equation}}
\newcommand{\ee}{\end{equation}}

\newcommand{\lagl}{\langle}
\newcommand{\ragl}{\rangle}
\newcommand{\lmx}{\left(\begin{matrix}}
\newcommand{\rmx}{\end{matrix}\right)}
\newcommand{\ldt}{\left|\begin{matrix}}
\newcommand{\rdt}{\end{matrix}\right|}

\newcommand{\dv}{{\rm div\,}}

\newcommand{\vfi}{\varphi}

\newcommand{\bbr}{{\mathbb R}}

\newcommand{\kr}{K\"ahler }

\newcommand{\bbc}{{\mathbb C}}

\newcommand{\ba}{\begin{array}}
\newcommand{\ea}{\end{array}}
\newcommand{\nnm}{\nonumber}
\newcommand{\beal}{\begin{align}}
\newcommand{\eal}{\end{align}}
\newcommand{\bea}{\begin{eqnarray}}
\newcommand{\eea}{\end{eqnarray}}

\newcommand{\pp}[2]{\fr{\partial #1}{\partial #2}}

\begin{document}

\title[On the Lagrangian angle and K\"ahler angle of surfaces in $\bbc^2$]{On the Lagrangian angle and the K\"ahler angle \\ of immersed surfaces in the complex plane $\bbc^2$} 

\author[X. X. Li]{Xingxiao Li$^*$} 

\author[X. Li]{Xiao Li} 

\dedicatory{}

%
\thanks{Research supported by
National Natural Science Foundation of China (No. 11171091, 11371018).}
\address{
School of Mathematics and Information Sciences
\endgraf Henan Normal University \endgraf Xinxiang 453007, Henan
\endgraf P.R. China}
\email{xxl$@$henannu.edu.cn}

\address{
School of Mathematics and Information Sciences
\endgraf Henan Normal University \endgraf Xinxiang 453007, Henan
\endgraf P.R. China} %
\email{lxlixiaolx$@$163.com}



\begin{abstract}
In this paper, we discuss the Lagrangian angle and the \kr angle of immersed surfaces in $\mathbb C^2$. Firstly, we provide an extension of Lagrangian angle, Maslov form and Maslov class to more general surfaces in $\mathbb C^2$ than Lagrangian surfaces, and then naturally extend a theorem by J.-M. Morvan to surfaces of constant \kr angle, together with an application showing that the Maslov class of a compact self-shrinker surface with constant \kr angle is generally non-vanishing. Secondly, we obtain two pinching results for the \kr angle which imply rigidity theorems of self-shrinkers with \kr angle under the condition that $\int_M |h|^2e^{-\frac{|x|^2}{2}}dV_M<\infty$, where $h$ and $x$ denote, respectively, the second fundamental form and the position vector of the surface.
\end{abstract}

\maketitle

\section{Introduction} 

Suppose that $(N,\mathbf{J})$ is an (almost) Hermitian manifold with $\omega$ its \kr form. Then the Riemannian metric $\langle\cdot,\cdot\rangle$ on $N$ is related with $\omega$ by
\be\langle U,V\rangle=\omega(U,\mathbf JV),\quad U, V\in TN.\label{1.1}\ee
Let $x:M^m \to  N$ be an immersed submanifold with tangent space $TM^m$, normal space $T^\bot M^m$ and volume form $dV_M$ of the induced metric. Then we call $x$ to be totally real if $J(x_*TM^m)\subset T^\bot M^m$; if, furthermore, $\dim_{\bbc}N=m$, then $x$ is called Lagrangian. In the special case of $m=2$, we can introduce the concept of \kr angle to define alternatively the totally real surfaces and Lagrangian surfaces. According to \cite{cw}, the \kr angle $\theta$ of an immersed surface $x:M^2\to N$ is defined by
\be x^*\omega=\cos\theta dV_M,\quad \theta\in [0,\pi].\label{1.2}\ee
Then $x$ is totally real if and only if $\cos\theta\equiv 0$, or equivalently, the \kr angle $\theta$ of $x$ is identical to $\fr\pi2$; $x$ is Lagrangian in $N$ if and only if $\theta\equiv\fr\pi2$ and $\dim_{\bbc}N=2$. By the way, if $\cos\theta>0$ everywhere, then $x$ is called symplectic.

The \kr angle of conformal minimal surfaces in complex projective space $\bbc P^n$ has been extensively studied. For example, each of the maps in the Veronese sequence gives a minimal immersion of $2$-sphere into $\bbc P^n$ that has a constant \kr angle (see \cite{bolton etc}, \cite{zhli}, \cite{lmwood}). The \kr angle of surfaces immersed in the nearly \kr manifold ${\mathbb S}^6$ is studied, for example, in \cite{ll} and \cite{lxx}.

Now put $N=\bbc^m$, the complex Euclidean space with complex coordinates $(z^1,\cdots,z^m)$. Let $\Omega=dz^1\wedge\cdots\wedge dz^m$ be the globally defined {\em holomorphic volume form}. For a Lagrangian submanifold $x:M^m\to \bbc^m$, the Lagrangian angle of $x$ is by definition a multi-valued function $\beta_x:M^m\to  \mathbb{R}/2\pi\mathbb{Z}$ given by
$$\Omega_M:=x^*\Omega=e^{\ima\beta_x}dV_M.$$
Note that, although the Lagrangian angle $\beta_x$ can not be determined globally in general, its gradient $\nabla \beta_x$ is clearly a well-defined vector field on $M^m$.

In \cite{mo}, J.-M. Morvan proved an important formula by which the mean curvature and the Lagrangian angle of a Lagrangian submanifold are linked to each other. Precisely we have

\begin{thm}[\cite{mo}]\label{mo} Let $x:M^m\to \bbc^m$ be a Lagrangian submanifold and $\mathbf{J}$ be the canonical complex structure of $\bbc^m$. Then the mean curvature vector $H$ and the Lagrangian angle $\beta_x$ meet the following formula:
\be\label{1.3} H=\mathbf{J}x_*(\nabla\beta_x),\ee
where $\nabla$ denotes the gradient operator on $M^m$.\end{thm}

Consequently, $x$ is minimal if and only if its {\em Maslov form} $\alpha:=-d\beta_x$ vanishes identically, implying that $x$ has a trivial {\em Maslov class} $[\alpha]\in H^1(M)$, where $H^1(M)$ is the first homology group of $M^m$.

\begin{rmk}\rm Formulas similar or related to \eqref{1.3} have also been obtained by some other authors, say, in \cite{c-g}, \cite{s-w} and etc. In particular, A. Arsie has extended Theorem \ref{mo} in \cite{ar} to Lagrangian submanifolds in general Calabi-Yau manifolds.\end{rmk}

Theorem \ref{mo} and its various extensions have many applications and citations in a lot of literatures, including papers on the Lagrangian mean curvature flow. As mentioned in \cite{cl2}, there is a direct interesting application of \eqref{1.3} first given by Smoczyk, and then by Castro and Lermain (\cite{cl2}) using a different method: {\em No compact and orientable Lagrangian self-shrinkers to the Lagrangian mean curvature flow in $\bbc^m$ with trivial Maslov class}.

In this paper, we consider the Lagrangian angle and \kr angle of immersed surfaces in $\bbc^2$. As the first part, we provide an extension of Lagrangian angle, Maslov form and Maslov class to more general surfaces in $\mathbb C^2$ than Lagrangian surfaces, and then naturally extend Theorem \ref{mo} of J.-M. Morvan to surfaces of constant \kr angle, together with an application showing that the Maslov class of a compact self-shrinker surface with constant \kr angle is generally non-vanishing. Precisely, we shall prove

\begin{thm}\label{thm1.2}
Let $x:M^2\to \bbc^2$ be an immersed surface with constant \kr angle $\theta$. Then the mean curvature vector $H$, the canonical complex structure $\mathbf{J}$ of $\bbc^2$ and the Lagrangian angle $\beta_x$ meet the following formula:
\be\label{1.4}\sin^2\theta x_*(\nabla\beta_x)=-(\mathbf{J}H)^\top,\ee
where $^\top$ denotes the projection onto the tangent space $x_*(TM^2)$.\end{thm}

Clearly, if $x$ is a Lagrangian surface then \eqref{1.4} agrees with \eqref{1.3}.

Note that an immersion $x:M^m\to\bbr^{m+p}$ is called a {\em self-shrinker} (of the mean curvature flow) if $H=-x^\bot $, where $^\bot$ denotes the orthogonal projection to the normal bundle of $x$. Thus an application of Theorem \ref{thm1.2} proves the following

\begin{thm}\label{thm1.3}
Let $x:M^2\rightarrow \bbc^2$ be a self-shrinker with constant \kr angle $\theta\in(0,\pi)$. If M is compact and orientable, then the Maslov class $[\alpha]\neq 0$.\end{thm}

\begin{cor}\label{cor1.2} There are no self-shrinkers with constant \kr angle $\theta\in(0,\pi)$ in $\bbc^2$ with the topology of a sphere.\end{cor}

As for symplectic surfaces there has been a fairly large amount of work. In particular, if the initial surface of a mean curvature flow in a \kr\!-Einstein surface is symplectic, then the surface $M_t$ at every time $t$ is also symplectic (\cite{cl1}, \cite{ct1}). So this kind of flow is referred to as a {\em symplectic mean curvature flow}. In \cite{as}, by using the self-adjoint property of a stability operator, Arezzo-Sun prove the following gap theorem (restated for the different definition of self-shrinkers here):

\begin{thm}[\cite{as}] Suppose that $x:M^2\to \bbc^2$ is a complete symplectic self-shrinker with polynomial volume growth. If $|h|^2\leq 2$, then $|h|^2\equiv 0$ and $x(M^2)$ must be a plane.\end{thm}

Moreover, Han-Sun prove a theorem in \cite{hs} that a translating soliton to the symplectic mean curvature flow in $\bbc^2$ with polynomial volume growth, non-positive normal curvature and bounded second fundamental form must be minimal in case that $\cos\theta\geq \delta>0$. The following is a similar result for symplectic self-shrinkers:

\begin{thm}[\cite{as}] Let $x:M^2\to \bbc^2$ be a complete symplectic self-shrinker with \kr angle $\theta$ and polynomial volume growth. If $|h|^2$ is bounded and $\cos\theta\geq\delta>0$, then $x(M^2)$ must be a plane.\end{thm}

Recently Li-Wang proved the following theorem which can be viewed as a complement of the above theorems:

\begin{thm}[\cite{lw1}]\label{thm1.7} Let $x:M^2\rightarrow \bbc^2$ be a compact orientable Lagrangian self-shrinker, and $h$ be the second fundamental form of $x$. If $|h|^2\leq2$, then $|h|^2\equiv 2$ and $x(M^2)$ is the Clifford torus $S^1(1)\times S^1(1)$.
\end{thm}

\begin{rmk}\rm Castro and Lerma also proved Theorem \ref{thm1.7} in \cite{cl2} under the additional condition that the Gauss curvature $K$ of $M^{2}$ is either non-negative or non-positive.\end{rmk}

To extend the above theorems, we provide in this paper the following main results:

\begin{thm}\label{thm1.8} Let $x:M^2\to \bbc^2$ be a complete self-shrinker with \kr angle $\theta$ satisfying  $\cos\theta\geq0$. If
$$\int_M|h|^2e^{-\frac{|x|^2}{2}}dV_M<\infty,$$
then $\theta$ is constant and $x(M^2)$ is either a Lagrangian surface or a plane.\end{thm}

To state the next main theorem, we need to introduce, as is done in \cite{cl1} and \cite{ct2}, an almost complex structure $\mathbf J_M$ on $\bbc^2$ along $x$. Let $\{e_1,e_2\}$ be an oriented orthonormal frame field on $M^2$ of which $\{\omega^i\}$ is the dual frame, and $\{e_{3},e_{4}\}$ be an oriented orthonormal normal frame field of $x$. Define
$$\mathbf J_M(x_*e_1)=x_*e_2,\quad  \mathbf J_M(x_*e_2)=-x_*e_1,\
\mathbf J_Me_3=-e_4,\quad \mathbf J_Me_4=e_3.$$
Then, if we write
$$h=\sum_{i,j}\left(h^3_{ij}\omega^i\omega^je_3 +h^4_{ij}\omega^i\omega^je_4\right),$$
then it holds that
\begin{align*}
(\bar D_{e_i}\mathbf J_M)(x_*e_1)=&\bar D_{e_i}(x_*e_2)-\mathbf J_M\bar D_{e_i}(x_*e_1)\\
=&-(h_{1i}^4-h_{2i}^3)e_3+(h_{2i}^4+h_{1i}^3)e_4,\\
(\bar D_{e_i}\mathbf J_M)(x_*e_2)=&-\bar D_{e_i}(x_*e_1)-\mathbf J_M\bar D_{e_i}(x_*e_2)\\
=&-(h_{2i}^4+h_{1i}^3)e_3-(h_{1i}^4-h_{2i}^3)e_4,\\
(\bar D_{e_i}\mathbf J_M)(e_3)=&-\bar D_{e_i}e_4-\mathbf J_M\bar D_{e_i}e_3\\
=&(h_{1i}^4-h_{2i}^3)(x_*e_1)+(h_{2i}^4+h_{1i}^3)(x_*e_2),\\
(\bar D_{e_i}\mathbf J_M)(e_4)=&\bar D_{e_i}e_3-\mathbf J_M\bar D_{e_i}e_4,\\
=&-(h_{2i}^4+h_{1i}^3)(x_*e_1)+(h_{1i}^4-h_{2i}^3)(x_*e_2).\\
\end{align*}
Consequently, we have
$$| \bar D \mathbf J_M|^2=4\sum_{i}\left((h_{2i}^4+h_{1i}^3)^2+(h_{1i}^4-h_{2i}^3)^2\right).$$

Note that $| \bar D \mathbf J_M|^2$ is globally defined on $M^2$ and thus is free of the choice of $\{e_1,e_2\}$ and $\{e_3,e_4\}$.

\begin{thm}\label{thm1.9} Let $x:M^2\to \bbc^2$ be a complete self-shrinker with \kr angle  $\theta$. If $$\int_M|h|^2e^{-\frac{|x|^2}{2}}dV_M<\infty$$
and there exists a number $\lambda\in [0,1)$ such that
\be |\nabla\theta|^2\leq\frac{\lambda\cos^2\theta|\bar D J_M|^2}{4(1-\lambda\cos^2\theta)}\label{a},\ee then $\theta$ is constant and $x(M^2)$ is either a Lagrangian surface or a plane.\end{thm}

\begin{cor}\label{cor1.10}
Let $x:M^2\to \bbc^2$ be a compact orientable self-shrinker with \kr angle $\theta$ satisfying either $\cos\theta\geq 0$ or
$$
|\nabla\theta|^2\leq\frac{\lambda\cos^2\theta|\bar D J_M|^2}{4(1-\lambda\cos^2\theta)}
$$
for some $\lambda\in [0,1)$. If $|h|^2\leq2$, then
$$\theta=\frac{\pi}{2},\quad |h|^2=2$$
and $x(M^2)$ is the Clifford torus $S^1(1)\times S^1(1)$.
\end{cor}

{\em Proof of Corollary \ref{cor1.10}}\,: Since $M^2$ is compact, by Theorem \ref{thm1.8} or Theorem \ref{thm1.9}, we know that $x:M^2\to\bbc^2$ is a Lagrangian surface. Then it follows from Theorem \ref{thm1.7} that $x(M^2)$ must be the Clifford torus $S^1(1)\times S^1(1)$.\endproof

\section{\kr angles of surfaces and self-shrinkers of mean curvature flow}

In this section, we recall some necessary basics and formulas about the \kr angle and self-shrinkers. First we make the following convention for the ranges of indices which are to be agreed with, if no other specification, throughout this present paper:
$$1\leq i, j, k\leq2,\quad 3\leq a, b, c\leq4,\quad 1\leq A, B, C\leq4.$$

On $\bbc^2$ we use the standard complex coordinates
$$z_1=x_1+\ima y_1,\quad z_2=x_2+\ima y_2.$$
The canonical complex structure $\mathbf{J}$ on $\bbc^2$ is given by
$$\mathbf J\frac{\partial}{\partial x_i}=\frac{\partial}{\partial y_i},\quad \mathbf J\frac{\partial}{\partial y_i}=-\frac{\partial}{\partial x_i}.$$

Let $x: M^2 \to\bbc^2$ be an immersed surface and $\omega$, $\langle\cdot,\cdot\rangle$ be the \kr form and the Riemannian metric on $\bbc^2$, respectively. The \kr angle $\theta$ of $x$ in $\bbc^2$ is defined by \eqref{1.2} which is equivalent to \be
\cos\theta=\lagl \mathbf J(x_*e_1),x_*e_2\ragl,\label{2.1}
\ee where $\{e_1,e_2\}$ is an oriented orthonormal frame field on $M^2$.

Starting from any given $\{e_1,e_2\}$ on $M^2$, we can construct along $x$ an orthonormal frame field $\{x_*e_1,x_*e_2,e_{3},e_{4}\}$ of $\bbc^2$, such that the following are satisfied in case $\theta\neq 0,\pi$ \cite{ll}:
\be\begin{aligned}&\mathbf{J}(x_*e_1)=(x_*e_2)\cos\theta +e_{3}\sin\theta,\quad
\mathbf{J}(x_*e_2)=-(x_*e_1)\cos\theta +e_{4}\sin\theta,\\
&\mathbf{J}(e_{3})=-(x_*e_1)\sin\theta -e_{4}\cos\theta,\quad
\mathbf{J}(e_{4})=-(x_*e_2)\sin\theta +e_{3}\cos\theta.\end{aligned}\label{2.2}
\ee
Denote by $D$, $\bar D$ and $D^\bot$ the Levi-Civita connections on $M^2$, $\bbc^2$ and the normal bundle, respectively. Then the formulas of Gauss and Weingarten are given by
$$\bar D_X(x_*Y)= x_*(D_XY)+h(X,Y),\quad \bar D_X\xi=-x_*(A_\xi X)+ D^\bot_X\xi,$$
where $X$ and $Y$ are tangent vector fields, $\xi$ is a normal vector field on $x$, $h$ denotes the second fundamental form and $A$ is the shape operator.

In what follows we identify $M^2$ with $x(M^2)$ and omit $x_*$ from some formulas and equations.

Let $\{\omega^A\}$ be the dual frame field of $\{e_A\}$, and $\{\omega_A^B\}$ the components of the Levi-Civita connection of $\bbc^2$ with respect to $\{e_A\}$. Then a direct computation similar to that in \cite{ll} using \eqref{2.2} shows

\begin{lem}[cf. \cite{ll}]\label{ch} If $\sin\theta\neq0$, then the following identities hold for an immersed surface $x:M^2\to \bbc^2$:
\be d\theta=\omega_1^4-\omega_2^3,\label{2.3}\ee
\be(\omega_1^3+\omega_2^4)\cos\theta+(\omega_3^4-\omega_1^2)\sin\theta=0.\label{2.4}\ee\end{lem}

Note that $\omega_i^a=\sum_j h_{ij}^a\omega^j$, where $h_{ij}^a$ are components of the second fundamental form $h$ of $x$. Then by \eqref{2.3}, the \kr angle $\theta$ of $x$ is constant if and only if
\be\label{rch}
h_{11}^4=h_{12}^3,\quad h_{12}^4=h_{22}^3.
\ee

Now we turn to self-shrinkers of the mean curvature flow in Euclidean space.

Let $x:M^m\to \mathbb{R}^{m+p}$ be an immersed $n$-dimensional submanifold in the $(m+p)$-dimensional Euclidean space with the mean curvature vector $H$. Then $x$ is called a {\em self-shrinker} (of the mean curvature flow) if $H=-x^\bot $, where $^\bot$ denotes the orthogonal projection to the normal bundle of $x$.

For a self-shrinker $x:M^m\to \mathbb{R}^{m+p}$, there is an important operator $\mathcal{L}$ acting on smooth functions that was first introduced and used by Colding and Minicozzi (\cite{cm}) to study self-shrinkers. The definition of $\mathcal{L}$ is as follows:
\be
\mathcal{L}=\Delta-\lagl x,\nabla\cdot\ragl=e^{\frac{|x|^2}{2}}\dv(e^{-\frac{|x|^2}{2}} \nabla\cdot),\ee where $\Delta$, $\nabla$ and $\dv$ denote the Laplacian, gradient and divergence on $M^m$, respectively.

Given an orthonormal tangent frame field $\{e_i;\ 1\leq i\leq m\}$ on $M^m$ with the dual $\{\omega^i\}$ and a normal frame field $\{e_\alpha;\ m+1\leq \alpha\leq m+p\}$, write $h=\sum_{\alpha,i,j}h^\alpha_{ij}\omega^i\omega^je_\alpha$. Then the mean curvature vector is by definition
$H=\sum_\alpha H^\alpha e_\alpha$ with $H^\alpha=\sum_ih^\alpha_{ii}$.

The following lemma is fundamental for study of self-shrinkers.

\begin{lem} [\cite{cl}]\label{s1} Let $x:M^m\to \mathbb{R}^{m+p}$ be a self-shrinker. Then
\be H_{,i}^\alpha=\sum_j h_{ij}^\alpha\lagl x,e_j\ragl.\label{2.20}\ee
\end{lem}

We end this section with another lemma which is essential to our argument.

\begin{lem} [\cite{lw2}]\label{s2} Let $x: M^m\to \mathbb{R}^{m+p}$ be a complete immersed submanifold. If $u$ and $v$ are $C^2$-smooth functions with
$$\int_M(|u\nabla v|+|\nabla u\nabla v|+|u\mathcal{L} v|)e^{-\frac{|x|^2}{2}}dV_M< \infty,$$
then it holds that
$$\int_M u\mathcal{L}v e^{-\frac{|x|^2}{2}}dV_M=-\int_M\lagl \nabla u, \nabla v\ragl e^{-\frac{|x|^2}{2}}dV_M.$$\end{lem}

\section{The Lagrangian angle for immersed surfaces in $\bbc^2$}

In this section, we first define a natural extension of the concept of Lagrangian angles to general immersed surfaces in $\bbc^2$, and then provide a proof of Theorem \ref{thm1.2}.

Given such a surface $x:M^2\to\bbc^2$ with an orthonormal tangent frame field $\{e_1,e_2\}$ and its dual $\{\omega^1,\omega^2\}$. As mentioned in the introduction, let $\Omega$ denote the holomorphic volume form on $\bbc^2$ given by $\Omega=dz_1\wedge dz_2$, where $z_1=x_1+\ima y_1$, $z_2=x_2+\ima y_2$. 

\begin{lem}\label{lem3.1}
When pulling back by $x$, $\Omega$ is a complex multiple of the volume form $dV_M$. More precisely, there exists a complex function $\eta_x$ such that $x^*\Omega=\eta_x dV_M$.
\end{lem}

\proof In fact, $\eta_x=x^*\Omega(e_1,e_2)=\Omega(x_*e_1,x_*e_2)$ since $dV_M=\omega^1\wedge\omega^2$.\endproof

Therefore, when $x^*\Omega\neq 0$, or the same, $\eta_x\neq 0$, there determined (up to a multiple of $2\pi$) is a multi-valued angle function $\beta_x$, usually called the argument of $\eta_x$, by $\eta_x=|\eta_x|e^{\ima \beta_x}$.

Thus the following definition is natural:

\begin{dfn}\rm Let $x: M^2\to \bbc^2$ be an immersed surface with $x^*\Omega\neq 0$. Then the multi-valued angle function $\beta_x$ is called the Lagrangian angle of $x$.\end{dfn}

\begin{rmk}\rm
Clearly, $\alpha:=-d\beta_x$ is a well defined closed $1$-form on $M^2$ which is called {\em the Maslov form} of $x$ and represents a cohomology class $[\alpha]\in H^1(M^2)$, referred to as {\em the Maslov class} of $x$.
\end{rmk}

Next we are to find how the \kr angle $\theta$ influences the Lagrangian angle $\beta_x$ for an immersion $x:M^2\to\bbc^2$. To this end, we need the following simple lemma:

\begin{lem}For \label{asu}any $A\in SU(2)$, it holds that $\eta_{Ax}=\eta_x$. In particular, $\beta_{Ax}=\beta_x$ in case that $\eta_x\neq 0$.\end{lem}

\proof Note that
$$SU(2)=\left\{\lmx a&b\\ -\bar{b}&\bar{a} \rmx\Big|a,b\in\bbc,|a|^2+|b|^2=1\right\}.$$
Write $z=(z_1,z_2)\in \bbc^2$ and, for any $A=\lmx a&b\\ -\bar b&\bar a\rmx \in SU(2)$, $\td z\equiv Az=(\td z_1,\td z_2)$. Since $\det A=1$, we find
\be\label{3.1} A^*\Omega=d\td z_1\wedge d\td z_2=(\det A)dz_1\wedge dz_2=\Omega.\ee

On the other hand, by Lemma \ref{lem3.1}, we have
$$x^*\Omega=\eta_x dV_M,\quad (Ax)^*\Omega=\eta_{Ax}dV_M.$$
Hence by \eqref{3.1}
$$
\eta_{Ax}=((Ax)^*\Omega)(e_1,e_2)=(x^*(A^*\Omega))(e_1,e_2)=(x^*\Omega)(e_1,e_2)=\eta_x.
$$
\endproof

Now we suppose that the \kr angle $\theta\in (0,\pi)$. For any $p\in M^2$, we can find an $A\in SU(2)$ to rotate $\bbc^2$ such that $(Ax)_*e_1=\frac{\partial}{\partial x_1}$ at $p$. Then, since $(Ax)_*e_1$ and $(Ax)_*e_2$ are orthogonal, we obtain from \eqref{2.1} that
$$(Ax)_*e_2=\cos\theta\frac{\partial}{\partial y_1}+\sin\theta(\cos\beta_{Ax}\frac{\partial}{\partial x_2}+\sin\beta_{Ax}\frac{\partial}{\partial y_2}).$$
Since
\be\Omega=(dx_1+\ima dy_1)\wedge(dx_2+\ima dy_2),\label{2.5}\ee
we compute
$$\Omega((Ax)_*e_1,(Ax)_*e_2)=\sin\theta(\cos\beta_{Ax}+\ima\sin\beta_{Ax})=\sin\theta e^{\ima\beta_{Ax}}.$$
Hence
\be (Ax)^*\Omega=\sin\theta e^{\ima\beta_{Ax}}dV_M,\ \mb{ at }p.\label{2.6}\ee
According to Lemma \ref{asu}, we can get
\be x^*\Omega=\sin\theta e^{\ima\beta_{x}}dV_M,\ \mb{ at }p.\label{2.6z}\ee
Thus we have proved
\begin{prop}
Let $x:M^2\to \bbc^2$ be an immersed surface with \kr angle $\theta\in (0,\pi)$. Then
$|\eta_x|^2=\sin^2\theta$. In particular, $\eta_x\neq 0$ everywhere on $M^2$.
\end{prop}


We are in a position to give a proof of Theorem \ref{thm1.2}.

First we assume that $\sin\theta\neq 0$.

Note that, for any orthonormal frame field $\{e_1,e_2\}$, we can find a unique normal frame $\{e_3,e_4\}$ by \eqref{2.2}. So, along $x(M^2)$, we have an orthonormal frame $\{x_*e_i,e_a\}$ of $\bbc^2$ with the dual frame field $\{\omega^A,1\leq A\leq 4\}$. As the first step to prove Theorem \ref{thm1.2}, we claim that
\be\label{OME}
\Omega\circ x=\frac{1}{\sin\theta}e^{\ima\beta_x} (\omega^1-\ima\mathbf{J}\omega^1)\wedge(\omega^2-\ima\mathbf{J}\omega^2 ).
\ee

In fact, since $\Omega$ is a $(2,0)$-form and $\omega^1-\ima\mathbf{J}\omega^1$, $\omega^2-\ima\mathbf{J}\omega^2$ are two linearly independent $(1,0)$-forms, we can write
$$\Omega\circ x=\mu\left(\omega^1-\ima\mathbf{J}\omega^1\right) \wedge\left(\omega^2-\ima\mathbf{J}\omega^2\right)$$
for some complex-valued function $\mu$ on $M^2$.

Since by \eqref{2.2}
\be\label{3.6}\begin{aligned}\mathbf{J}(\omega^{1})=&-\omega^2\cos\theta -\omega^{3}\sin\theta,\quad \mathbf{J}(\omega^2)=\omega^{1}\cos\theta -\omega^{4}\sin\theta,\\
\mathbf{J}(\omega^{3})=&\omega^{1}\sin\theta +\omega^{4}\cos\theta,\qquad \mathbf{J}(\omega^{4})=\omega^2\sin\theta -\omega^{3}\cos\theta,
\end{aligned}\ee
it follows that
\begin{align}
&\fr12(\omega^i-\ima\mathbf{J}\omega^i)(x_*e_i-\ima\mathbf{J}(x_*e_i))=1,\quad i=1,2;\label{3.7}\\
&\fr12(\omega^1-\ima\mathbf{J}\omega^1)(x_*e_2-\ima\mathbf{J}(x_*e_2))\nnm\\ &\qquad =-\fr12(\omega^2-\ima\mathbf{J}\omega^2)(x_*e_1-\ima\mathbf{J}(x_*e_1))\nnm\\
&\qquad =\ima\cos\theta.\label{3.8}
\end{align}
So that
\be \Omega\left(\frac{x_*e_1-\ima\mathbf{J}(x_*e_1)}{2},\frac{x_*e_2 -\ima\mathbf{J}(x_*e_2)}{2}\right)=\mu \sin^2\theta.\label{o2}\ee

On the other hand, for any given $p\in M^2$, we can assume that $x_{*p}e_1=\frac{\partial}{\partial x_1}|_{x(p)}$ because $A^*\Omega=\Omega$ for any $A\in SU(2)$. Then it holds that
\begin{align} x_*e_1-\ima&\mathbf J (x_*e_1)=\frac{\partial}{\partial x_1}-\ima\frac{\partial}{\partial y_1},\label{2.7}\\
x_*e_2-\ima&\mathbf J(x_*e_2)=\cos\theta\frac{\partial}{\partial y_1}+\sin\theta\left(\cos\beta_x\frac{\partial}{\partial x_2}+\sin\beta_x\frac{\partial}{\partial y_2}\right)\nnm\\
&+\ima\cos\theta\frac{\partial}{\partial x_1}-\ima\sin\theta\left(\cos\beta_x\frac{\partial}{\partial y_2}-\sin\beta_x\frac{\partial}{\partial x_2}\right)\label{2.8}\end{align}
at $p$.
Substitute \eqref{2.7} and \eqref{2.8} into \eqref{2.5} and obtain
\be \Omega\left(\frac{x_*e_1-\ima\mathbf{J}(x_*e_1)}{2},\frac{x_*e_2 -\ima\mathbf{J}(x_*e_2)}{2}\right)=\sin\theta e^{\ima\beta_x}\label{o1}\quad\mb{at }p.\ee
Comparing \eqref{o2} with \eqref{o1}, we obtain $$\mu(p)=\frac{1}{\sin\theta}e^{\ima\beta_x(p)}.$$
Thus \eqref{OME} is proved by the arbitrariness of the point $p$.

The second step is to make use of the fact that $\Omega$ is parallel on $\bbc^2$. For any tangent vector field $X\in TM^2$, we have
\be\bar{ D}_X\Omega=0.\label{2.9}\ee
Note that the \kr angle $\theta$ is constant. \eqref{2.9} is equivelant to
\begin{align}
0=&\ima\sin^2\theta d\beta_x(X)\Omega+\sin\theta e^{\ima\beta_x}(\bar{ D}_X\omega^1-\ima\bar{ D}_X\mathbf{J}\omega^1)\wedge(\omega^2-\ima \mathbf{J}\omega^2)\nnm\\
&+\sin\theta e^{\ima\beta_x}(\omega^1-\ima\mathbf{J}\omega^1)\wedge(\bar{ D}_X\omega^2-\ima\bar{ D}_X\mathbf{J}\omega^2).\label{2.10}
\end{align}

Since the Levi-Civita connection $\bar{D}$ is also a complex one, that is, $\bar D\mathbf{J}=\mathbf{J}\bar D$, we know that
$\bar{ D}_X\omega^1-\ima \bar{ D}_X\mathbf{J}\omega^1$ is a linear combination of $$\omega^1-\ima\mathbf{J}\omega^1 \quad\rm{and}\quad \omega^2-\ima\mathbf{J}\omega^2,$$
or
\be\label{3.13}
\bar{D}_X\omega^1-\ima \bar{ D}_X\mathbf{J}\omega^1=\rho_1(\omega^1-\ima\mathbf{J}\omega^1) +\rho_2(\omega^2-\ima\mathbf{J}\omega^2)
\ee
for two complex-valued functions $\rho_1$ and $\rho_2$.

Moreover, from $\bar D\mathbf{J}=\mathbf{J}\bar D$ and \eqref{3.6}, we have
\begin{align*}
\omega^1(\bar{D}_X\mathbf{J}e_1)=&\omega^1(\mathbf{J}(\bar{D}_Xe_1))=(\mathbf J \omega^1)(\bar{D}_X e_1)\\
=&(-\omega^2\cos\theta -\omega^{3}\sin\theta)(\bar{D}_X e_1)\\
=&-\cos\theta\langle \bar{D}_X e_1,e_2 \rangle-\sin\theta\langle \bar{D}_X e_1,e_3 \rangle,\\
\omega^1(\bar{D}_X\mathbf{J}e_2)=&\omega^1(\mathbf{J}(\bar{D}_Xe_2))=(\mathbf J \omega^1)(\bar{D}_X e_2)\\
=&(-\omega^2\cos\theta -\omega^{3}\sin\theta)(\bar{D}_X e_2)
=-\sin\theta\langle \bar{D}_X e_2,e_3 \rangle.
\end{align*}
From \eqref{3.7}, \eqref{3.8} and \eqref{3.13} it follows that
\begin{align}\rho_1-\ima \rho_2\cos\theta=&(\bar{D}_X\omega^1-\ima \bar{ D}_X\mathbf{J}\omega^1)((e_1-\ima\mathbf{J}e_1)/2)\notag\\
=&-\ima(\cos\theta\langle \bar{D}_X e_1,e_2 \rangle +\sin\theta\langle \bar{ D}_Xe_1,e_3 \rangle).\label{2.14}\\
\ima\rho_1\cos\theta+ \rho_2=&(\bar{D}_X\omega^1-\ima \bar{ D}_X\mathbf{J}\omega^1)((e_2-\ima\mathbf{J}e_2)/2)\notag\\=&-\ima\sin\theta\langle \bar{D}_X e_2,e_3 \rangle-\langle \bar{D}_X e_2,e_1 \rangle.\label{2.14a}\end{align}
Combining\eqref{2.14} and \eqref{2.14a} we can find $\rho_1$ and $\rho_2$ to obtain
\begin{align}
&\sin\theta(\bar{D}_X\omega^1-\ima \bar{ D}_X\mathbf{J}\omega^1)\nnm\\
=&\left(\cos\theta \langle \bar{D}_X e_2,e_3 \rangle-\ima\langle \bar{D}_X e_1,e_3 \rangle\right)(\omega^1-\ima\mathbf{J}\omega^1)\nnm\\
&-\left(\sin\theta\langle \bar{D}_X e_2,e_1 \rangle +\ima\langle \bar{D}_X e_2,e_3 \rangle +\cos\theta \langle \bar{D}_X e_1,e_3 \rangle\right)(\omega^2-\ima\mathbf{J}\omega^2)\end{align}
Similarly we have
\begin{align}&\sin\theta(\bar{D}_X\omega^2-\ima \bar{ D}_X\mathbf{J}\omega^2)\nnm\\
=&\left(-\cos\theta \langle \bar{D}_X e_1,e_4 \rangle-\ima\langle \bar{D}_X e_2,e_4 \rangle\right)(\omega^2-\ima\mathbf{J}\omega^2)\notag\\
&-\left(\sin\theta\langle \bar{D}_X e_1,e_2 \rangle-\cos\theta \langle \bar{D}_X e_2,e_4 \rangle +\ima\langle \bar{D}_X e_1,e_4 \rangle\right)(\omega^1-\ima\mathbf{J}\omega^1).\end{align}
Therefore we can rewrite \eqref{2.10} as
\begin{align}
0=&\left(\ima\sin\theta d\beta_x(X)+\cos\theta \langle \bar{D}_X e_2,e_3 \rangle-\ima \langle \bar{D}_X e_1,e_3 \rangle\right. \notag\\
 &\left.-\cos\theta \langle \bar{D}_X e_1,e_4 \rangle-\ima \langle \bar{D}_X e_2,e_4 \rangle\right)\Omega.\label{2.11}
\end{align}
Using Gauss formula and \eqref{2.11}, we find
\begin{align}
\ima\sin\theta d\beta_x(X)=&-\cos\theta \langle \bar{D}_X e_2,e_3 \rangle+\ima\langle \bar{D}_X e_1,e_3 \rangle \notag\\
 &+\cos\theta \langle \bar{D}_X e_1,e_4 \rangle+\ima\langle \bar{D}_X e_2,e_4 \rangle\notag\\
 =&-\cos\theta \langle h(X,e_2),e_3 \rangle+\ima\langle h(X,e_1),e_3 \rangle \notag\\
 &+\cos\theta \langle h(X,e_1),e_4 \rangle+\ima\langle h(X,e_2),e_4 \rangle\notag\\
 =&-\cos\theta \langle \bar{D}_{e_2} X ,e_3 \rangle+\ima\langle \bar{D}_{e_1} X ,e_3 \rangle \notag\\
 &+\cos\theta \langle \bar{D}_{e_1} X ,e_4 \rangle+\ima\langle \bar{D}_{e_2} X ,e_4 \rangle\notag\\
 =&\cos\theta \langle \bar{D}_{e_2} e_3,X \rangle-\ima\langle \bar{D}_{e_1} e_3, X \rangle\notag \\
 &-\cos\theta \langle \bar{D}_{e_1} e_4,X \rangle-\ima\langle \bar{D}_{e_2} e_4,X \rangle\label{2.12b}\end{align}
But from \eqref{2.2} it is easily seen that
\be\label{2.12} e_3=\frac{1}{\sin\theta}\mathbf Je_1-\frac{\cos\theta}{\sin\theta}e_2,\quad
e_4=\frac{1}{\sin\theta}\mathbf Je_2+\frac{\cos\theta}{\sin\theta}e_1.
\ee
Inserting \eqref{2.12} to \eqref{2.12b}, we obtain
\begin{align*}\ima\sin^2\theta d\beta_x(X)=&\cos\theta \big\langle \bar{D}_{e_2} \big(\mathbf Je_1\big),X \big\rangle-\cos\theta \big\langle \bar{D}_{e_2} \big(\cos\theta e_2\big),X \big\rangle\notag\\&-\ima\big\langle \bar{D}_{e_1} \big(\mathbf Je_1\big), X \big\rangle+\ima\big\langle \bar{D}_{e_1} \big(\cos\theta e_2\big), X \big\rangle \notag\\
&-\cos\theta \big\langle \bar{D}_{e_1} \big(\mathbf Je_2\big),X \big\rangle-\cos\theta \big\langle \bar{D}_{e_1} \big(\cos\theta e_1\big),X \big\rangle\\
&-\ima\big\langle \bar{D}_{e_2} \big(\mathbf Je_2\big),X \big\rangle-\ima\big\langle \bar{D}_{e_2} \big(\cos\theta e_1\big),X \big\rangle\\
=&\cos\theta \langle \mathbf J{D}_{e_2} e_1,X \rangle+\cos\theta \langle \mathbf Jh(e_2,e_1),X \rangle-\cos^2\theta \langle D_{e_2}e_2,X \rangle\\
&-\ima\langle \mathbf J{D}_{e_1}e_1, X \rangle-\ima\langle \mathbf Jh(e_1,e_1), X \rangle+\ima\cos\theta \langle{D}_{e_1} e_2, X \rangle \\
&-\cos\theta \langle \mathbf J{D}_{e_1}e_2,X \rangle-\cos\theta \langle \mathbf Jh(e_1,e_2),X \rangle-\cos^2\theta \langle {D}_{e_1} e_1,X \rangle\\
&-\ima\langle \mathbf J{D}_{e_2}e_2,X \rangle-\ima\langle \mathbf Jh(e_2,e_2),X \rangle-\ima\cos\theta \langle {D}_{e_2} e_1,X \rangle.
\end{align*}
For any given point $p\in M^2$, choose proper frame field $\{e_i\}$ satisfying $D_{e_i}e_j=0$ at $p$. Then
\begin{align*}
\sin^2\theta d\beta_x(X)=&-\langle \mathbf J(h(e_1,e_1)),X\rangle-\langle \mathbf J(h(e_2,e_2)),X\rangle\\
=&-\langle \mathbf J(h(e_1,e_1)+h(e_2,e_2)),X\rangle\\
=&-\langle (\mathbf JH)^\top,X\rangle\quad\mb{at }p.
\end{align*}
Since both $d\beta_x(X)$ and $\langle (\mathbf JH)^\top,X\rangle$ are globally defined functions and thus free of the choice of the frame field $\{e_i\}$, and $X$ is an arbitrary vector field on $M^2$, we have $\sin^2\theta\nabla\beta_x=-(\mathbf{J}H)^\top$ and Theorem \ref{thm1.2} is proved for $\sin\theta\neq 0$.

When $\sin\theta=0$, both sides of \eqref{1.4} vanish identically since any holomorphic curve in $\bbc^2$ is necessarily minimal.\endproof

\section{Examples}

In this section, we provide three examples of surfaces in $\bbc^2$. Example \ref{expl3.1} and \ref{expl3.2} are quoted from lecture notes by Jason Lotay.

{\expl\label{expl3.1}\rm Let $x: \mathbb{R}^2\to \bbc^2$ be given by
$$ x(u_1,u_2)=(e^{\ima u_1},e^{\ima u_2}),\quad ( u_1, u_2)\in \bbr^2.$$
A short calculation shows that the immersion is an immersed Lagrangian self-shrinker with the image $x(\bbr^2)=S^1(1)\times S^1(1)$, a standard torus. Furthermore, if we denote $e_1=\pp{}{ u_1}$, $e_2=\pp{}{ u_2}$, then
$$x_*e_1=(\ima e^{\ima  u_1},0),\quad x_*e_2=(0,\ima e^{\ima  u_2}).$$It is easy to check that
$$(x^*\Omega)(e_1,e_2)=e^{\ima( u_1+ u_2+\pi)}$$
implying that the Lagrangian angle of $x$ is given by $\beta_x= u_1+ u_2+\pi$ up to multiples of $2\pi$. Obviously, the Maslov class $[-d\beta_x]$ of $x$ is nontrivial and, by Theorem \ref{thm1.2} $x$ is not minimal.}

{\expl\label{expl3.2}\rm Let $x: \mathbb{R}^2\backslash\{0\}\to \bbc^2$ be given by
$(z_1,z_2)= x( u_1, u_2)$ with
$$z_1=u_1+ \frac{ u_1}{ u_1^2+ u_2^2}\ima,\quad
z_2= u_2+\frac{ u_2}{ u_1^2+ u_2^2}\ima,\quad ( u_1, u_2)\in \bbr^2.$$
This is known as the Lagrangian catenoid and, with respect to the induced metric $g$, an orthonormal frame field can be chosen as
$$e_1=\fr{u^2_1+u^2_2}{\sqrt{1+(u^2_1+u^2_2)^2}}\pp{}{u_1},\quad e_2=\fr{u^2_1+u^2_2}{\sqrt{1+(u^2_1+u^2_2)^2}}\pp{}{u_2}.$$
Then
$$dV_M=\fr{1+(u^2_1+u^2_2)^2}{(u^2_1+u^2_2)^2}du_1\wedge du_2,\quad\mb{ where } M^2=(\bbr^2,g),$$
and
\begin{align*}x_*e_1=&\fr{u^2_1+u^2_2}{\sqrt{1+(u^2_1+u^2_2)^2}}\left(1,0,\frac{ u_2^2- u_1^2}{( u_1^2+ u_2^2)^2},\frac{-2 u_1  u_2}{( u_1^2+ u_2^2)^2}\right),\\
x_*e_2=&\fr{u^2_1+u^2_2}{\sqrt{1+(u^2_1+u^2_2)^2}}\left(0,1,\frac{-2 u_1 u_2}{( u_1^2+ u_2^2)^2},\frac{ u_1^2- u_2^2}{( u_1^2+ u_2^2)^2}\right).
\end{align*}
A direct computation shows that
$$x^*\Omega=x^*(dz_1\wedge dz_2)=\fr{1+(u^2_1+u^2_2)^2}{(u^2_1+u^2_2)^2}du_1\wedge du_2=dV_M.$$
Thus
$$\eta_x\equiv 1 \quad \mb{and}\quad {\rm Arg}(\eta_x)=2k\pi,\ k\in\mathbb{Z}.$$
Therefore, the Lagrangian angle $\beta_x=0$ up to multiples of $2\pi$. Consequently by Theorem \ref{thm1.2}, $x$ is a minimal immersion.}

{\expl\label{expl3.3}\rm For any given two real constants $\theta_1$ and $\theta_2$, denote $\theta:=\theta_1+\theta_2$. Let $x: \mathbb{R}^2\to \bbc^2$ be defined by
$(z_1,z_2)= x( u_1, u_2)$ with
$$z_1=u_1\cos\theta_1+\ima u_2\cos\theta_2,\quad z_2=- u_2\sin\theta_2- \ima u_1\sin\theta_1,\quad ( u_1, u_2)\in \bbr^2.$$
Choose $e_1=\pp{}{u_1}$ and $e_2=\pp{}{u_2}$. Then
\be\label{e1}x_*e_1=(\cos\theta_1,0,0,-\sin\theta_1),\quad x_*e_2=(0,-\sin\theta_2,\cos\theta_2,0).\ee
Thus
\be\label{e3}\mathbf{J}(x_*e_1)=(0,\sin\theta_1,\cos\theta_1,0),\quad \mathbf{J}(x_*e_2)=(-\cos\theta_2,0,0,-\sin\theta_2).\ee
It follows that
$$\lagl \mathbf J(x_*e_1),x_*e_2\ragl=\cos(\theta_1+\theta_2)=\cos\theta,$$
implying that $x$ is of constant \kr angle $\theta$.

On the other hand, since
\begin{align*}
x^*\Omega=x^*(dz_1 \wedge dz_2)=&(\cos\theta_1du_1+\ima \cos\theta_2du_2)\wedge (-\sin\theta_2du_2- \ima\sin\theta_1 du_1)\\
=&-\sin(\theta_1+\theta_2)du_1\wedge du_2 =-\sin(\theta_1+\theta_2)dV_{\bbr},
\end{align*}
we know that the Lagrangian angle of $x$ is
$\beta_x=\pi$ up to multiples of $2\pi$. So by Theorem \ref{thm1.2}, $H=0$. Clearly, $x(\bbr^2)$ is actually a $2$-plane in $\bbc^2$ and is indeed totally geodesic.}

\section{Proof of the other main theorems}

In this section, we give proofs of other main theorems.

(1) {\em Proof of Theorem \ref{thm1.3}.}

For any tangent vector $v\in TM^2$, by the definition of a self-shrinker and the fundamental equations we find
$$
-A_Hv+D^\bot_v H
=\bar{D}_vH=\bar{D}_v(-x^\bot)
=\bar{D}_v(x^\top-x)\\
=D_vx^\top+h(v,x^\top)-v.
$$
So that
\be A_Hv=v- D_vx^\top,\quad D^\bot_vH=h(v,x^\top).\label{5.1}\ee
For a given orthonormal frame field $\{e_1,e_2\}$ on $M^2$, let the normal frame $\{e_3,e_4\}$ be defined via \eqref{2.2}. Then a direct computation shows that
\be
\mathbf{J}H=-\sin\theta (H^3x_*e_1+H^4x_*e_2)+\cos\theta(H^4e_3-H^3e_4)\label{5.2}
\ee
which implies that
\begin{align}(\mathbf JH)^\top=&-\sin\theta(H^3x_*e_1+H^4x_*e_2),\label{5.3}\\
(\mathbf JH)^\bot=&\cos\theta(H^4e_3-H^3e_4).\label{5.4}\end{align}
Since the \kr angle $\theta$ is constant, by Lemma \ref{ch} and \eqref{rch}, we know that
\be h_{11}^4=h_{12}^3,\quad h_{12}^4=h_{22}^3,\quad \omega_1^2=\omega_3^4+\cot\theta(\omega_1^3+\omega_2^4).\label{5.5}\ee
Applying \eqref{5.3} -- \eqref{5.5}, we have
\begin{align*}
 D_{e_i}(\mathbf JH)^\top=&-\sin\theta D_{e_i}(H^3e_1+H^4e_2)\\
=&-\sin\theta\big(e_i(H^3)e_1+e_i(H^4)e_2+H^3\omega_1^2(e_i)e_2+H^4\omega_2^1(e_i)e_1 \big)\\
=&-\sin\theta\big(e_i(H^3)e_1+e_i(H^4)e_2+(H^3e_2-H^4e_1)\omega_1^2(e_i)\big)\\
=&-\sin\theta\big(e_i(H^3)e_1+e_i(H^4)e_2+(H^3e_2-H^4e_1)\omega_3^4(e_i)\\
&+\cot\theta(H^3e_2-H^4e_1)(\omega_1^3+\omega_2^4)(e_i)\big)\\
=&-\sin\theta\big(e_i(H^3)e_1+e_i(H^4)e_2+(H^3e_2-H^4e_1)\omega_3^4(e_i)\big)\\
&-\cos\theta(H^3e_2-H^4e_1)(h_{1i}^3+h_{2i}^4);\end{align*}
and
\begin{align*}
\mathbf J D_{e_i}^\bot H=&\mathbf J D_{e_i}^\bot(H^3e_3+H^4e_4)\\
=&\mathbf J\big(e_i(H^3)e_3+e_i(H^4)e_4+H^3 D_{e_i}^\bot e_3+H^4 D_{e_i}^\bot e_4\big)\\
=&\mathbf J\big(e_i(H^3)e_3+e_i(H^4)e_4+H^3\omega_3^4(e_i)e_4+H^4\omega_4^3(e_i)e_3\big)\\
=&\big(e_i(H^3)-H^4\omega_3^4(e_i)\big)(-\sin\theta e_1-\cos\theta e_4)\\
&+\big(e_i(H^4)+H^3\omega_3^4(e_i)\big)(-\sin\theta e_2+\cos\theta e_3).
\end{align*}
Comparing the above two equations, we find
\be D_{e_i}(\mathbf JH)^\top=(\mathbf J D_{e_i}^\bot H)^\top-\cos\theta(H^3e_2-H^4e_1)(h_{1i}^3+h_{2i}^4).\label{5.6}\ee
Consequently, from \eqref{5.1}, \eqref{5.5} and \eqref{5.6} it follows that
\begin{align*}
\dv(\mathbf JH)^\top=&\sum_{i}\lagl  D_{e_i}(\mathbf JH)^\top,e_i\ragl\\
=&\sum_{i}\lagl\mathbf J( D_{e_i}^\bot H)-\cos\theta(H^3e_2-H^4e_1)(h_{1i}^3+h_{2i}^4),e_i\ragl\\
=&\sum_{i}\lagl\mathbf J( D_{e_i}^\bot H),e_i\ragl-\cos\theta(\lagl(H^3e_2-H^4e_1)(h_{11}^3+h_{21}^4),e_1\ragl\\
&+\lagl(H^3e_2-H^4e_1)(h_{12}^3+h_{22}^4),e_2\ragl)\\
=&\sum_{i}\lagl\mathbf J( D_{e_i}^\bot H),e_i\ragl-\cos\theta(-H^4(h_{11}^3+h_{22}^3)+H^3(h_{11}^4+h_{22}^4))\\
=&\sum_{i}\lagl\mathbf J( D_{e_i}^\bot H),e_i\ragl-\cos\theta(-H^4H^3+H^3H^4)\\
=&-\sum_{i}\lagl( D_{e_i}^\bot H),\mathbf J e_i\ragl=-\sum_{i}\lagl h(e_i,x^\top),\mathbf J e_i\ragl\\
=&-\lagl h(e_1,x^\top),\sin\theta e_3\ragl-\lagl h(e_2,x^\top),\sin\theta e_4\ragl\\
=&-\lagl h(e_1,x^1e_1+x^2e_2),\sin\theta e_3\ragl-\lagl h(e_2,x^1e_1+x^2e_2),\sin\theta e_4\ragl\\
=&-\sin\theta(x^1h_{11}^3+x^2h_{12}^3+x^1h_{21}^4+x^2h_{22}^4)\\
=&-\sin\theta(x^1(h_{11}^3+h_{22}^3)+x^2(h_{11}^4+h_{22}^4))\\
=&-\sin\theta(x^1H^3+x^2H^4)\\
=&\lagl (\mathbf JH)^\top,x^\top\ragl,\notag
\end{align*}
where $x^\top=\sum_i x^ie_i$.
If $[\alpha]=0$, there exists a globally defined smooth function $\vfi$ satisfying $\alpha=-d\vfi$, or equivalently, $\nabla\vfi=\nabla\beta_x$. By Theorem \ref{thm1.2}, we obtain an elliptic linear equation
$$\Delta\vfi=\frac{1}{2}\lagl\nabla\vfi,\nabla|x|^2\ragl.$$
Then according to the maximum principle (see for example, \cite{gt}), $\vfi$ must be a constant and thus $\alpha=-d\vfi\equiv 0$ which with Theorem \ref{thm1.2} implies $H\equiv0$, contradicting the fact that there are no compact minimal submanifolds in an Euclidean space.
\endproof



To prove Theorem \ref{thm1.8} and Theorem \ref{thm1.9}, the following lemma is needed:

\begin{lem}Let $x:M^2\to \bbc^2$ be a self-shrinker with \kr angle $\theta$. Then \begin{align}&\mathcal{L}\cos\theta=-\frac{1}{4}\cos\theta|\bar D \mathbf J_M|^2,\label{5.7}\\
\frac{1}{2}\mathcal{L}\cos^2&\theta=\sin^2\theta| \nabla\theta|^2-\frac{1}{4}\cos^2\theta|\bar D \mathbf J_M|^2.\label{5.8}\end{align}
\end{lem}

\begin{proof} Around each point of the open subset
$$M^*:=\{p\in M^2;\ \sin\theta(p)\neq0\},$$
we choose an orthonormal frame field $\{e_1,e_2\}$ on $M^2$ which via \eqref{2.2} defines the normal frame $\{e_3,e_4\}$. Then by \eqref{2.3} and \eqref{2.4}
\be \nabla\theta=\sum_i(h_{1i}^4-h_{2i}^3)e_i,\quad \Gamma_{1i}^2-\Gamma_{3i}^4=\cot\theta(h_{1i}^3+h_{2i}^4),\label{5.9}\ee
where $\Gamma_{AB}^C$ are determined by $\omega_A^C=\sum_B\Gamma^C_{AB}\omega^B$. From Codazzi equation and \eqref{5.9}, it follows that
\begin{align}
\Delta\theta=&\sum_i(h_{1i}^4-h_{2i}^3)_{,i}\notag \\
=&\sum_i(e_i(h_{1i}^4-h_{2i}^3)-(h_{1j}^4-h_{2j}^3)\Gamma_{ii}^j)\nonumber\\
=&\sum_i(h_{1ii}^4-h_{2ii}^3+h_{2i}^4\Gamma_{1i}^2-h_{1i}^3\Gamma_{2i}^1-h_{1i}^3\Gamma_{3i}^4+h_{2i}^4\Gamma_{4i}^3)\nonumber\\
=&\sum_i(H_{,1}^4-H_{,2}^3+(h_{2i}^4+h_{1i}^3)\Gamma_{1i}^2-(h_{2i}^4+h_{1i}^3)\Gamma_{3i}^4)\nonumber\\
=&\sum_i(H_{,1}^4-H_{,2}^3+(h_{2i}^4+h_{1i}^3)(\Gamma_{1i}^2-\Gamma_{3i}^4))\nonumber\\
=&\sum_i(H_{,1}^4-H_{,2}^3+\cot\theta(h_{2i}^4+h_{1i}^3)^2).\label{5.11}
\end{align}
Consequently, by \eqref{5.9} and \eqref{5.11},
\begin{align}
\Delta\cos\theta=&-\cos\theta|\nabla\theta|^2-\sin\theta\Delta\theta\notag\\
=&-\cos\theta\sum_i(h_{1i}^4-h_{2i}^3)^2-\sin\theta(H_{,1}^4-H_{,2}^3)-\cos\theta\sum_i(h_{2i}^4+h_{1i}^3)^2\notag\\
=&-\cos\theta(\sum_i((h_{1i}^4-h_{2i}^3)^2+(h_{2i}^4+h_{1i}^3)^2))-\sin\theta(H_{,1}^4-H_{,2}^3)\notag\\
=&-\frac{1}{4}\cos\theta|\bar D \mathbf J_M|^2-\sin\theta(H_{,1}^4-H_{,2}^3),\label{5.12}
\end{align}
and thus
\begin{align}
\frac{1}{2}\Delta(\cos^2\theta)=&|\nabla\cos\theta|^2+\cos\theta(\Delta\cos\theta)\notag\\
=&\sin^2\theta|\nabla\theta|^2-\frac{1}{4}\cos^2\theta|\bar D \mathbf J_M|^2-\cos\theta\sin\theta(H_{,1}^4-H_{,2}^3).\label{5.13}
\end{align}
By Lemma \ref{s1}, we have
\be
H_{,1}^4-H_{,2}^3=\sum_i h_{1i}^4\lagl x,e_i\ragl-\sum_i h_{2i}^3\lagl x,e_i\ragl
=\sum_i(h_{1i}^4-h_{2i}^3)\lagl x,e_i\ragl.
\label{5.14}
\ee
Inserting \eqref{5.14} to \eqref{5.12} and \eqref{5.13}, respectively, we obtain
\begin{align}
\Delta\cos\theta=&-\frac{1}{4}\cos\theta|\bar D \mathbf J_M|^2-\sin\theta\sum_i(h_{1i}^4-h_{2i}^3)\lagl x,e_i\ragl\notag\\
=&-\frac{1}{4}\cos\theta|\bar D \mathbf J_M|^2-\lagl x,\sin\theta\sum_i(h_{1i}^4-h_{2i}^3)e_i\ragl\notag\\
=&-\frac{1}{4}\cos\theta|\bar D \mathbf J_M|^2+\lagl x, \nabla\cos\theta\ragl,\label{5.15}
\end{align}
and
\begin{align}
\frac{1}{2}\Delta\cos^2\theta=&|\nabla\cos\theta|^2+\cos\theta(\Delta\cos\theta)\notag\\
=&\sin^2\theta|\nabla\theta|^2-\frac{1}{4}\cos^2\theta|\bar D \mathbf J_M|^2+\cos\theta\lagl x, \nabla\cos\theta\ragl\notag\\
=&\sin^2\theta|\nabla\theta|^2-\frac{1}{4}\cos^2\theta|\bar D \mathbf J_M|^2+\cos\theta\lagl x, \nabla\cos\theta\ragl\notag\\
=&\sin^2\theta|\nabla\theta|^2-\frac{1}{4}\cos^2\theta|\bar D \mathbf J_M|^2+\frac{1}{2}\lagl x, \nabla\cos^2\theta\ragl.\label{5.16}
\end{align}
By the definition of $\mathcal{L}$, \eqref{5.15} and \eqref{5.16} become, respectively,
\begin{align}
\mathcal{L}&\cos\theta=\Delta\cos\theta-\lagl x, \nabla\cos\theta\ragl
=-\frac{1}{4}\cos\theta|\bar D \mathbf J_M|^2,\label{5.17}\\
\frac{1}{2}\mathcal{L}\cos^2\theta=&\frac{1}{2}\Delta\cos^2\theta-\frac{1}{2}\lagl x, \nabla\cos^2\theta\ragl
=\sin^2\theta| \nabla\theta|^2-\frac{1}{4}\cos^2\theta|\bar D \mathbf J_M|^2.\label{5.18}
\end{align}
Thus \eqref{5.7} and \eqref{5.8} hold on the open set $M^*$.

At a point $p\not\in M^*$, there are two cases that need to be considered:

Case $1^\circ$: $p$ is an accumulation point of $M^*$.

Because both $\mathcal{L}\cos\theta$ and $-\cos\theta|\bar D \mathbf J_M|^2$ are globally defined and smooth, we can take the limits of the two sides of \eqref{5.7} to find \eqref{5.7} still holds at the point $p$. Similarly \eqref{5.8} holds at $p$.

Case $2^\circ$: $p$ is not an accumulation point of $M^*$.

In this case, $\mathbf J_M=\mathbf J$ around $p$. So $|\bar D \mathbf J_M|^2=0$.
Since $\cos\theta\equiv 1$ around $p$, $\mathcal{L}\cos\theta=0$. Therefore both \eqref{5.7} and \eqref{5.8} still hold at $p$.
\end{proof}

(2) {\em Proof of Theorem \ref{thm1.8}}

Take $v=\cos\theta$. Then
\begin{align}
|\nabla v|^2=&\sin^2\theta|\nabla \theta|^2=\sin^2\theta\sum_i|h_{1i}^4-h_{2i}^3|^2 \nnm\\
\leq& 2\sin^2\theta\sum_i((h_{1i}^4)^2+(h_{2i}^3)^2)\leq 2|h|^2,\label{5.19}\\
|\mathcal{L}v|=&|\mathcal{L}\cos\theta|
=\frac{1}{4}|\cos\theta|\bar D \mathbf J_M|^2|\leq4|h|^2.\label{5.20}
\end{align}

Since $\int_M|h|^2e^{-\frac{|x|^2}{2}}dV_M<\infty$, \eqref{5.19} and \eqref{5.20} show that, for $u\equiv 1$, $$\int_M(|u\nabla v|+|\nabla u\nabla v|+|u\mathcal{L} v|)e^{-\frac{|x|^2}{2}}dV_M< \infty.$$
Thus by Lemma \ref{s2}, we obtain
\be\int_M\mathcal{L}\cos\theta e^{-\frac{|x|^2}{2}}dV_M=0.\label{5.21}\ee
Since, by the assumption, $\cos\theta\geq 0$, it follows by \eqref{5.7} that $\mathcal{L}\cos\theta\leq0$. Therefore,
\be\mathcal{L}\cos\theta=-\frac{1}{4}\cos\theta|\bar D \mathbf J_M|^2\equiv0.\label{id}\ee

Consider the function $\varphi=|\nabla\theta|^2$.
If $\varphi$ is not identically zero on $M^2$, then there exists a point $p_0$ such that $\varphi\neq0$ and we can find a connected domain $U$ containing $p_0$ where $\varphi$ is non-zero identically. Denote $U_0=\{p\in U|\cos\theta(p)=0\}$. Then $U_0$ obviously contains no interior points. Therefore $U\backslash U_0$ is a dense set. By \eqref{id}, $|\bar D \mathbf J_M|^2\equiv0$ on $U\backslash U_0$ implying $h_{1i}^4=h_{2i}^3$ and $h_{2i}^4=-h_{1i}^3$ on $U\backslash U_0$. Thus, by \eqref{rch}, we have $\nabla\theta\equiv0$ on $U\backslash U_0$ which contradicts the definition of $U$. This shows that $\varphi\equiv0$ on $M^2$, that is, $\theta$ is constant.

If $\cos\theta=0$, then $x$ is a Lagrangian immersion;

If $\cos\theta\neq 0$, then $|\bar D \mathbf J_M|^2\equiv0$. It follows that $h_{1i}^4=h_{2i}^3$, $h_{2i}^4=-h_{1i}^3$, that is,
$$h_{11}^4=h_{21}^3=-h_{22}^4,\quad  h_{22}^3=h_{21}^4=-h_{11}^3$$
which implies that $H\equiv 0$.
Since $x$ is a self-shrinker, we obtain
\be\lagl x,e_3\ragl=0,\quad \lagl x,e_4\ragl=0.\label{5.22}\ee
Hence, for $a=3,4$,
$$\lagl x,\bar  D_{e_i}e_a\ragl=e_i\lagl x,e_a\ragl-\lagl\bar D_{e_i}x,e_a\ragl=0-\lagl e_i,e_a\ragl=0,$$
namely
$$\lagl x,-A_{e_a}e_i+\sum_b\Gamma_{ai}^be_b\ragl=0.$$
So that
\be \sum_jh_{ij}^a\lagl x,e_j\ragl=0,\quad a=3,4.\label{5.23}\ee
From \eqref{5.23}, it is easily seen that
$$
\ldt
 h_{11}^3 & h_{21}^3 \\
 h_{12}^3 & h_{22}^3 \\
\rdt=\ldt
 h_{11}^4 & h_{21}^4 \\
 h_{12}^4 & h_{22}^4 \\
\rdt=0.
$$
It follows that $$h_{11}^3h_{22}^3-(h_{21}^3)^2=h_{11}^4h_{22}^4-(h_{21}^4)^2=0$$
which with $H=0$ shows that
$$h_{11}^3=h_{21}^3=h_{11}^4=h_{21}^4=0.$$
That is, $M^2$ is totally geodesic and thus, by the completeness, $M^2$ must be a plane.
\endproof

(3) {\em Proof of Theorem \ref{thm1.9}}

Similar to those in (2), we take $u\equiv 1$ and $v=\frac{1}{2}\cos^2\theta$. Since
\begin{align}
|\nabla v|^2=&\cos^2\theta\sin^2\theta|\nabla\theta|^2
=\cos^2\sin^2\theta\sum_i|h_{1i}^4-h_{2i}^3|^2\notag\\
\leq& 2\sin^2\theta\sum_i((h_{1i}^4)^2+(h_{2i}^3)^2)
\leq 2|h|^2,\label{5.25}\\
|\mathcal{L}v|=&\frac{1}{2}|\mathcal{L}\cos^2\theta|
=|\sin^2\theta|\nabla\theta|^2-\frac{1}{4}\cos^2\theta|\bar D \mathbf J_M|^2|
\leq6|h|^2,\label{5.26}
\end{align}
we can use Lemma \ref{s2} to get
\be\frac{1}{2}\int_M(\mathcal{L}\cos^2\theta) e^{-\frac{|x|^2}{2}}dV_M=0.\label{5.27}\ee
On the other hand, by \eqref{a} and \eqref{5.8}, we find
\begin{align*}
\frac{1}{2}\mathcal{L}\cos^2\theta=&\sin^2\theta| \nabla\theta|^2-\frac{1}{4}\cos^2\theta|\bar D \mathbf J_M|^2\\
\leq&\sin^2\theta\fr{\lambda\cos^2\theta|\bar D \mathbf J_M|^2}{4(1-\lambda\cos^2\theta)}-\frac{1}{4}\cos^2\theta|\bar D \mathbf J_M|^2\\
=&-\fr{(1-\lambda)\cos^4\theta|\bar D \mathbf J_M|^2}{4(1-\lambda\cos^2\theta)}
\leq 0.
\end{align*}
Therefore
$$\cos^4\theta|\bar D \mathbf J_M|^2=0.$$

By using an argument similar to that in (2), we can conclude that either $\cos\theta\equiv 0$ or $|\bar D \mathbf J_M|^2\equiv 0$. Then the rest of the proof is omitted here since it is the same as that of the proof of Theorem \ref{thm1.8}. \endproof

\end{document}